\documentclass[12pt]{article}
\usepackage{amsmath,amssymb,amsthm}
\usepackage{indentfirst}
\usepackage[margin=1in]{geometry}
\usepackage[hidelinks]{hyperref}

\newtheorem{theorem}{Theorem}[section]
\newtheorem{lemma}[theorem]{Lemma}

\newtheorem{corollary}[theorem]{Corollary}
\newtheorem{conjecture}[theorem]{Conjecture}
\theoremstyle{remark}

\DeclareMathOperator{\per}{per}
\DeclareMathOperator{\ret}{ret}
\DeclareMathOperator{\dist}{dist}

\newcommand{\cF}{\mathcal F}
\newcommand{\cB}{\mathcal B}
\allowdisplaybreaks[2]

\begin{document}

\title{The biwheel conjecture for braces\footnote{Research supported by the National Natural Science Foundation of
China (Grant No.\ 12671408).}}

\author{Xing Feng, Jinqiu Zhou\footnote{E-mail address:
jqzhougt@163.com (Corresponding author).}\\
School of Science, Jimei University, Xiamen 361021, China\\
}

\date{}

 \maketitle{}

\begin{abstract}
Let $b(n)$ denote the minimum number of perfect matchings in a simple
brace of order $n$. De~Carvalho, Lucchesi and Murty (SIAM Journal on 
Discrete Mathematics, 27(2): (2013) 940-958) proved that 
$b(n)\ge(n-2)^2/8$ and conjectured that $b(n)=(n-2)^2/4$ for all
sufficiently large even $n$. We prove this conjecture.
\end{abstract}



\section{Introduction}\label{sec:intro}

A connected graph with at least one edge is \emph{matching covered} if
each of its edges belongs to a perfect matching. An edge cut is 
\emph{tight} if every perfect matching meets it in exactly one edge. A 
bipartite matching covered graph is a \emph{brace} if it has no tight 
cut whose two shores both contain more than one vertex. We write $m(G)$ 
for the number of perfect matchings of a graph $G$.

Braces are the bipartite building blocks in tight cut decomposition. It 
is therefore natural to ask how few perfect matchings a brace of
given order can have. De Carvalho, Lucchesi and Murty~\cite[Theorem~5.3]{CLM13}
proved that every brace of order $n\ge4$ has at least $(n-2)^2/8$
perfect matchings. They also proposed the following sharper bound.

\begin{conjecture}[Biwheel conjecture]\label{conj:biwheel}
\textup{\cite[Conjecture~5.2]{CLM13}}
For all sufficiently large even integers $n$, the minimum number of
perfect matchings in a brace of order $n$ is $(n-2)^2/4$.
\end{conjecture}

For $r\ge4$, the \emph{biwheel} of order $2r$ is obtained from a cycle
of length $2r-2$ by adding two vertices, called its hubs, and joining
each hub to all vertices in one of the two colour classes of the cycle.
The hubs lie in opposite colour classes. A biwheel is a simple brace
and has exactly $(r-1)^2$ perfect matchings: choose one rim neighbour
for each hub; deleting these two opposite rim vertices leaves even
paths, each with a unique perfect matching. Adding the edge between its
hubs gives an \emph{extended biwheel}, with $(r-1)^2+2$ perfect matchings,
since matchings using the hub edge leave the two matchings of the rim.
Thus biwheels give the upper bound in Conjecture~\ref{conj:biwheel}.
The qualification concerning sufficiently large orders is necessary;
small braces can have fewer perfect matchings than a biwheel of the
same order~\cite[Section~5]{CLM13}.

Our main result is the following.

\begin{theorem}\label{thm:main}
There is an integer $r_0$ such that every simple brace $G$ of order
$2r$, where $r\ge r_0$, satisfies
\[
                         m(G)\ge(r-1)^2.
\]
Consequently, $b(2r)=(r-1)^2$ for all $r\ge r_0$.
\end{theorem}

Clearly, the restriction to simple graphs loses nothing for the lower 
bound. In particular, Theorem~\ref{thm:main} proves
Conjecture~\ref{conj:biwheel}. It also gives
\[
                  \lim_{r\longrightarrow\infty}
                  \frac{b(2r)}{(2r)^2}=\frac14.
\]

Throughout, constants are chosen for a transparent proof rather than
a useful numerical threshold. One admissible value is
\begin{equation}\label{eq:threshold}
 r_0=R^2,\qquad R=1+119\cdot69432\cdot7^{119}.
\end{equation}

\section{Preliminaries and linear estimates}\label{sec:prelim}

We follow the terminology of~\cite{BM08,LP09,LM24}.
Graphs are finite and loopless. Original braces will be simple;
auxiliary graphs may have parallel edges. Degrees, sizes and perfect
matchings distinguish all parallel copies. The \emph{rank} of a
bipartite graph with a perfect matching is the size of either colour
class. We write $G[A,B]$ to specify these classes. A vertex of degree
three is \emph{cubic}, and $x_+=\max\{x,0\}$.
An \emph{opposite pair} has one vertex in each colour class,
whether adjacent or not; its \emph{cofactor} is $m(G-a-b)$.
We set $m(\varnothing)=1$. A nontrivial parallel class has at
least two edge copies. Whenever a parallel class is used for a
reduction below, it is nontrivial. For a graph of rank $q$, its
\emph{edge excess} is $|E|-3q$, unless a different reference
quantity is explicitly specified; the degree excess of a vertex
means $d(v)-3$. Logarithms are natural and
$0\log0=0$ by continuity.

\begin{theorem}\label{thm:hall}
\textup{\cite{BM08}}
A bipartite graph $J[X,Y]$ has a matching covering $X$ if and only
if $|N_J(U)|\ge|U|$ for every $U\subseteq X$.
\end{theorem}

\begin{theorem}\label{thm:clmstructure}
\textup{\cite{CLM13, LP09, M0149}}
A connected bipartite graph $J[X,Y]$ with $|X|=|Y|=q\ge2$ is
matching covered if and only if $J-x-y$ has a perfect matching
for every $x\in X$, $y\in Y$, if and only if 
$|N_J(U)|\ge |U|+1$ ($\varnothing\ne U\subsetneq X$).
A bipartite matching covered graph $G[A,B]$ of rank $r\ge3$ is
a brace if and only if
\begin{equation}\label{eq:bracehall}
 |N_G(U)|\ge |U|+2\quad(1\le |U|\le r-2).
\end{equation}
If $G$ is a brace, every $G-a-b$, with $a\in A$, $b\in B$, is
matching covered.
\end{theorem}

\emph{Bicontracting} a degree-two vertex with distinct neighbours
means contracting its two
incident edges. Its two neighbours are merged, and no parallel copies
are discarded.

\begin{lemma}\label{lem:bicontraction}
\textup{\cite[Lemma~2.1]{CLM13}}
Bicontraction in a bipartite matching covered graph preserves
matching coverage and the number of perfect matchings.
\end{lemma}

Repeated bicontraction of $J$ gives a \emph{retract}, denoted by 
$\ret J$; fix a sequence when necessary. Each step lowers rank by 
one and size by two, replacing neighbour degrees $d_1,d_2$ by 
$d_1+d_2-2$. For a matching covered graph of minimum degree at 
least two, the process stops at a double edge or at minimum degree 
at least three.

Let $\cF$ be the class of bipartite matching covered multigraphs of
minimum degree at least three. An edge is \emph{removable} if its
deletion preserves matching coverage. A member $J$ is \emph{minimal
in $\cF$} if $J-e\notin\cF$ for every edge copy $e$.
An edge is \emph{solitary} if it belongs to exactly one perfect
matching. Let $\mu(J)$ be the largest multiplicity of a solitary
parallel class (possibly of multiplicity one), or zero if none exists.

\begin{theorem}\label{thm:clmlinear}
\textup{\cite[Theorems~4.1 and~4.11, Table~4.1]{CLM13}}
For $J\in\cF$ of rank $q$,
\begin{equation}\label{eq:clmlinear}
 m(J)\ge4q-8,\qquad m(J)\ge4q-4\quad(q\ge7).
\end{equation}
Moreover,
\begin{equation}\label{eq:clmmu}
                     m(J)\ge4q-4+\mu(J)\quad(q\ge7).
\end{equation}
\end{theorem}

\begin{lemma}\label{lem:clmminimal}
\textup{\cite[Corollary~4.3 and Lemma~4.4]{CLM13}}
Every nontrivial parallel class of a minimal member of $\cF$ has
a cubic endpoint. A simple minimal member has a cubic vertex
whose three incident edges are removable.
\end{lemma}

\begin{lemma}\label{lem:solitary}
\textup{\cite[Lemmas~4.7 and~4.9]{CLM13}}
In a member of $\cF$ with more than two vertices, each endpoint of
a solitary edge has a nontrivial parallel class to another vertex,
and two adjacent solitary edges are parallel.
\end{lemma}

The dimension bound \cite{L8772}, valid for every bipartite matching covered graph
$J$ of rank $q$, is
\begin{equation}\label{eq:dimension}
                         m(J)\ge |E(J)|-2q+2.
\end{equation}

An edge $e$ of a brace is \emph{thin} if $G-e$ is matching covered
and $\ret(G-e)$ is a brace, and
\emph{strictly thin} if this retract is simple. Its index is the number
of cubic endpoints, and belongs to $\{0,1,2\}$. For a strictly thin
edge of index $i$, the retract has rank $r-i$.

\begin{theorem}\label{thm:strictthin}
\textup{\cite{CLM15, M0149}}
A simple brace of order greater than four outside the families
of prisms, M\"obius ladders, biwheels and extended biwheels has a
strictly thin edge.
\end{theorem}
The first two families are cubic. We will repeatedly partition perfect 
matchings at a vertex. In a simple bipartite graph this gives
\[
                 m(G)=\sum_{b\in N_G(a)}m(G-a-b).
\]
In a multigraph the summands must instead be counted once for each 
incident edge copy.

\begin{lemma}\label{lem:linearaux}
Let $J\in\cF$ have rank $q$ and size $E$. Then
\begin{equation}\label{eq:edgeq}
                              m(J)\ge E-q.
\end{equation}
If $J$ is simple, then
\begin{equation}\label{eq:simplebounds}
                   m(J)\ge6q-24,\qquad m(J)\ge E-14.
\end{equation}
\end{lemma}
\begin{proof}
We first prove~\eqref{eq:edgeq} by induction on $E$. Rank one is
immediate. If $J$ is not minimal in $\cF$, delete an edge while
remaining in $\cF$ and observe that restoring it adds at least one
perfect matching.

If $J$ is minimal and nonsimple, Lemma~\ref{lem:clmminimal}
gives a cubic vertex $u$ with two edges to $v$ and one to $w\ne v$.
Three edges to the same vertex would force a rank-one component,
since every matching would pair these two vertices.
The set $\{u,v,w\}$ is a shore
of a tight cut: $u$ must be matched to one of $v,w$, and the other
vertex must be matched outside the set. Contracting it gives
$K\in\cF$ of rank $q-1$ and size $E-3$. In fact, every edge
of $K$ is the image of an edge in a perfect matching of $J$, and
the new vertex has degree $d(v)+d(w)-3\ge3$; all other degrees
are unchanged. Write $C_v,C_w$ for the cut edges originally incident
with $v,w$. A matching of $K$ using an edge of $C_v$ has one
extension, using $uw$, whereas one using an edge of $C_w$ has two,
using either copy of $uv$. Thus
\begin{equation}\label{eq:extension}
                    m(J)=m(K)+\sum_{f\in C_w}m_K(f),
\end{equation}
where $m_K(f)$ counts the matchings containing the specified copy
$f$. Since $K$ is matching covered, each summand is at least one.
Thus $|C_w|\ge2$ and the inductive bound
$m(K)\ge(E-3)-(q-1)=E-q-2$ prove~\eqref{eq:edgeq}.

If $J$ is simple and minimal, choose the cubic vertex supplied by
Lemma~\ref{lem:clmminimal}, with incident edges $e_1,e_2,e_3$. Each
$J-e_i$ has at most two degree-two vertices; they are nonadjacent,
and their bicontractions create no new degree-two vertex. Its retract
therefore has size $E-1-2j_i$ and rank $q-j_i$, where $j_i\le2$.
The nonadjacency follows from simplicity: the only possible centres
are the ends of the deleted edge. Identifying two vertices of
degree at least three creates degree at least four, so no further
centre appears. Each perfect matching of $J$ omits exactly two of
$e_1,e_2,e_3$. Thus induction gives
\[
 2m(J)=\sum_{i=1}^3m(J-e_i)\ge3(E-q-3).
\]
This suffices if $E-q\ge9$, by comparing $3(E-q-3)/2$ with $E-q$.
Otherwise $E\ge3q$ implies $q\le4$.
The only remaining possibilities are $K_{3,3}$ and the cubic graph
$K_{4,4}$ minus a perfect matching, with respectively six and nine
perfect matchings. Both satisfy~\eqref{eq:edgeq}.

Using~\eqref{eq:clmlinear} instead on the same three retracts gives
$2m(J)\ge3(4q-16)$ and hence the first inequality
in~\eqref{eq:simplebounds}. Passing to a minimal spanning member
proves it for every simple $J\in\cF$.

Every minimal member of $\cF$ has at most $5q-2$ edges. At rank one
the minimal graph consists of three parallel edges. For $q\ge2$,
fix a perfect matching, contract its edges, and orient the other
edges from the first colour class to the second. The resulting
digraph is strongly connected: matching coverage puts every
nonmatching edge on an alternating cycle, so every directed arc
lies on a directed cycle; connectivity then forces a single
strong component. Retain an in-arborescence and an
out-arborescence with a common root. Together with the matching
these give a spanning matching covered subgraph with at most $3q-2$
edges and minimum degree at least two. Indeed, their union is
strongly connected, and a directed cycle through any retained arc
gives a perfect matching using the corresponding edge. Every
vertex has both an incoming and an outgoing arc. Add at most one edge for
each vertex still of degree two, using at most $2q$ edges. Strong
connectivity, and hence matching coverage, is retained. Minimality
proves the size bound.

For a simple minimal $J$, combining $E\le5q-2$ with $m(J)\ge4q-8$
when $q\le8$, and with $m(J)\ge6q-24$ when $q\ge8$, gives
$E-m(J)\le14$. Restore deleted edges to obtain the second inequality
in~\eqref{eq:simplebounds} in general: each restored edge increases
the size by one and contributes at least one perfect matching
containing that edge and absent before its restoration. Indeed,
the spanning matching covered subgraph has a matching after the
ends of the restored edge are deleted, by
Theorem~\ref{thm:clmstructure}; adjoining that edge gives the new matching.
\end{proof}

For a bipartite multigraph $J[X,Y]$, let $b_X(J)$ count its cubic
vertices in $X$ incident with parallel edges.

\begin{lemma}\label{lem:potential}
If $J[X,Y]\in\cF$, then
\[
                         m(J)\ge |E(J)|-b_X(J)-14.
\]
\end{lemma}
\begin{proof}
Induct on the size. The simple case follows from
Lemma~\ref{lem:linearaux}, and rank one is immediate. Write
$\Phi(J)=|E(J)|-b_X(J)$. For a reduction to $K$, it suffices to
prove $m(J)-m(K)\ge\Phi(J)-\Phi(K)$, since the desired bound
then follows from induction for $K$. Suppose first
that no parallel class has a cubic endpoint. Deleting a parallel
copy preserves $\cF$: its endpoints retain degree at least three,
and any matching using the deleted copy can use a remaining
parallel copy. Only its $X$ endpoint can become newly counted by
$b_X$, so $\Phi$ decreases by at most two. A nonsolitary copy
belongs to at least two perfect matchings, which is enough for
the induction inequality. The only difficulty is a
solitary copy whose $X$ endpoint has degree four and becomes cubic
with parallel edges. That endpoint has a second parallel class, by
Lemma~\ref{lem:solitary}, and therefore has two double classes. A
copy in the other class is nonsolitary by the same lemma.
Delete that copy instead.

Otherwise take a cubic parallel endpoint $u$ and use the contraction
in~\eqref{eq:extension}. If $u\in X$, let $j$ be the number of
surviving cubic $X$ vertices newly incident with parallel edges.
Then $b_X(K)\le b_X(J)-1+j$. Each such vertex uses a distinct edge
of $C_w$: new parallelism requires an edge to each of $v,w$
before they are identified. This edge is nonsolitary in $K$:
otherwise another parallel class at the same cubic vertex would
be required, which is impossible at degree three. If $K$
has rank one, no newly counted vertex is possible. Hence the
increment in~\eqref{eq:extension} is at least $|C_w|+j\ge2+j$,
whereas $\Phi(J)-\Phi(K)=3-b_X(J)+b_X(K)\le2+j$.

Suppose $u\in Y$. Only the $X$ vertices $v,w$ are affected, and the
merged vertex has degree $d(v)+d(w)-3$. If both are noncubic, the
potential decreases by three and $|C_w|\ge3$. If $v$ is cubic and
$w$ noncubic, the decrease is two. If $w$ is cubic and has parallel
edges, the decrease is at most two. In the remaining cases $w$ is
cubic without parallel edges, and the decrease is at most three.
The two edges of $C_w$ then have distinct other endpoints; they
cannot both be solitary in $K$ by Lemma~\ref{lem:solitary}, so their matching counts sum to at
least three. Induction completes the proof.
\end{proof}

\begin{lemma}\label{lem:fibercharge}
Let $H[X,Y]$ be simple and matching covered, with minimum degree at
least two. Let $t$ be the number of its degree-two vertices in $X$.
For any $h\in Y$,
\[
 b_X(\ret H)\le2t+\sum_{y\in Y\setminus\{h\}}(d_H(y)-3)_+.
\]
\end{lemma}
\begin{proof}
Put $R=\ret H$ and fix a sequence of bicontractions
$H=H_0,H_1,\ldots,H_k=R$. We regard each bicontraction as deleting
a degree-two vertex $z$ and identifying its two neighbours $u,v$
to a new vertex $w$, retaining all parallel edges. The neighbours
are distinct unless the graph is already the terminal double edge.
This description gives the same graph as contracting both edges at
$z$, and makes the two colour classes explicit.

For each $v\in V(H_i)$ define a set $S_i(v)\subseteq V(H)$ as
follows. Initially $S_0(v)=\{v\}$. At the operation just described,
set $S_{i+1}(w)=S_i(u)\cup S_i(v)$, omit the set $S_i(z)$, and
leave all other sets unchanged. Thus the sets $S_i(v)$ are nonempty
and pairwise disjoint, and each lies entirely in $X$ or entirely in
$Y$. They record only the vertices in the same original colour
class that have been identified. The degree formula for an
identification, $d_{H_{i+1}}(w)=d_{H_i}(u)+d_{H_i}(v)-2$, gives
by induction
\[
                 d_{H_i}(v)-2
                    =\sum_{x\in S_i(v)}(d_H(x)-2).
\]
In particular, whenever a degree-two vertex $z$ is deleted, every
member of $S_i(z)$ had degree two in $H$.

Let $a_0$ be the number of vertices of $X$ in the sets omitted
when degree-two vertices on the $X$ side are deleted. These omitted
sets are disjoint, so $a_0\le t$. If $x$ is a cubic vertex on the
$X$ side of $R$, the displayed identity implies that $S_k(x)$
contains exactly one vertex of degree three in $H$, and all its
other members have degree two in $H$. Consequently, at most
$t-a_0$ such vertices $x$ satisfy $|S_k(x)|\ge2$.

For each vertex $y$ on the $Y$ side of $R$, choose
$r_y\in S_k(y)$, taking $r_y=h$ whenever $h\in S_k(y)$, and put
\[
                  U=\bigcup_{y\in Y(R)}
                             \bigl(S_k(y)\setminus\{r_y\}\bigr),
\]
where $Y(R)$ denotes the colour class of $R$ corresponding to $Y$.
Forming a set $S_k(y)$ requires $|S_k(y)|-1$ identifications on
the $Y$ side. Each of these deletes a degree-two vertex on the
$X$ side and omits a nonempty set counted in $a_0$. Hence
$|U|=\sum_{y\in Y(R)}(|S_k(y)|-1)\le a_0$, and $h\notin U$.
Every vertex of $U$ loses at least one of its original incident
edges when it first participates in an identification. Thus the
number of edges of $R$ whose original $Y$ endpoint belongs to $U$
is at most
\[
 \sum_{y\in U}(d_H(y)-1)
 \le 2|U|+\sum_{y\in U}(d_H(y)-3)_+
 \le 2a_0+\sum_{y\in Y\setminus\{h\}}(d_H(y)-3)_+.
\]

Now let $x$ be a cubic vertex on the $X$ side of $R$ incident
with parallel edges, with $S_k(x)=\{x_0\}$. Since $H$ is simple,
two parallel edges from $x$ to a vertex $y$ of $R$ have distinct
original endpoints in $S_k(y)$. At least one of these endpoints
lies in $U$. Choose one such edge for each $x$. Different vertices
$x$ give different edges, so the preceding bound counts all these
vertices. Adding the at most $t-a_0$ cubic vertices with
$|S_k(x)|\ge2$, we obtain
\[
 b_X(R)\le(t-a_0)+2a_0+
           \sum_{y\in Y\setminus\{h\}}(d_H(y)-3)_+
       \le2t+\sum_{y\in Y\setminus\{h\}}(d_H(y)-3)_+.
\]
This proves the lemma.
\end{proof}

\begin{lemma}\label{lem:cofactorbasic}
Let $G[A,B]$ be a simple brace of rank $r$ and size $s$, and let
$ab\in E(G)$ have endpoint degrees $d,k$. Then
\begin{align}
 m(G-a-b)&\ge4r-4d-4k-4,\label{eq:coarsecofactor}\\
 m(G-a-b)&\ge s-r-2d-2k+4,\label{eq:edgecofactor}\\
 m(G-a-b)&\ge s-d-k-2r+5.\label{eq:dimcofactor}
\end{align}
If the retract of $G-a-b$ has rank at least seven, the constant
$-4$ in~\eqref{eq:coarsecofactor} can be replaced by zero.
Moreover, for every vertex of degree $d$,
\begin{equation}\label{eq:dimensionsum}
 m(G)\ge(d-1)s-d^2-2rd+2d+3r.
\end{equation}
\end{lemma}
\begin{proof}
The graph $H=G-a-b$ is matching covered. Its degree-two vertices
must be cubic vertices of $G$ adjacent to one of the deleted
vertices, so there are at most $d+k-2$ of them. Each
bicontraction reduces their number by at least one: if both
neighbours of its centre have degree two, three such vertices are
replaced by one; otherwise no new degree-two vertex is created.
After $j\le d+k-2$ contractions the rank is $r-1-j$ and the size
is $s-d-k+1-2j$. Apply~\eqref{eq:clmlinear} and
\eqref{eq:edgeq} to this retract, whose matching count is still
$m(H)$, and replace $j$ by its upper bound in the resulting
decreasing expressions. A terminal double edge also satisfies
$m\ge4q-8$ and $m\ge E-q$ at $q=1$, $E=2$.
Equation~\eqref{eq:dimcofactor} is~\eqref{eq:dimension} applied
directly to $H$, with rank $r-1$ and size $s-d-k+1$.
Summing it over the neighbours of $a$ gives
$m(G)\ge d(s-d-2r+5)-\sum_{b\in N(a)}d(b)$.
The $r-d$ vertices outside $N(a)$
each have degree at least three, so
$\sum_{b\in N(a)}d(b)\le s-3(r-d)$. Its coefficient in the
summed bound is negative, which gives
\eqref{eq:dimensionsum}.
\end{proof}

\section{Unit cofactors and bounded excess}\label{sec:unit}

An opposite pair $(a,b)$ in a bipartite graph $J$ is a \emph{unit
pair} if $m(J-a-b)=1$. It need not be an edge.

\begin{lemma}\label{lem:unitidentity}
Let $J\in\cF$ have rank $q$ and a unit pair $(a,b)$. Let $\ell$
be the multiplicity of $ab$, and put $e=m(J)-\ell-4(q-1)$. Then
$e\ge0$ and
\begin{equation}\label{eq:nonhub}
            \sum_{v\notin\{a,b\}}(d_J(v)-3)\le e/2.
\end{equation}
Every edge of the unique perfect matching of $J-a-b$ belongs to at
least $m(J)-(e+4)^2$ perfect matchings of $J$. If
$B=m(J)-4q+4$, then
\begin{equation}\label{eq:hubdegree}
               d_J(a),d_J(b)\ge q-1-\lfloor B/4\rfloor.
\end{equation}
\end{lemma}
\begin{proof}
Write the unique matching as $\{x_i y_i:1\le i\le n\}$, where
$n=q-1$, with $x_i$ on the $a$ side. Its edges have multiplicity
one. For vertices $u,v$ of $J$, let $E_J(u,v)$ denote the set of
edge copies joining them. Define
\[
 \alpha_i=|E_J(a,y_i)|,\qquad
 \beta_i=|E_J(x_i,b)|,\qquad
 w_{ij}=|E_J(x_i,y_j)|\quad(i\ne j),
\]
and set $w_{ii}=0$. These numbers are zero when the corresponding
vertices are nonadjacent.

Construct a directed graph with vertices $S,T,1,\ldots,n$. For
each positive $\alpha_i$, include an arc $S\to i$ of weight
$\alpha_i$; for each positive $\beta_i$, include an arc $i\to T$
of weight $\beta_i$; and for each positive $w_{ij}$, include an
arc $i\to j$ of weight $w_{ij}$. Thus an arc weight is the number
of parallel edge copies in $J$ represented by that arc. The source
$S$ and sink $T$ themselves are not assigned weights; in particular,
$\alpha_i$ is the weight of the arc from $S$ to $i$, not a weight
of $S$. The $\ell$ copies of $ab$ are counted separately, and no
arc $S\to T$ is included. The internal digraph is acyclic, since
a directed cycle would give an alternating cycle with respect to
the unique matching of $J-a-b$.

Let
\[
             k_i=\sum_{j=1}^n w_{ji},\qquad
             l_i=\sum_{j=1}^n w_{ij}.
\]
These are the sums of the weights of the internal arcs entering
and leaving $i$, respectively. Since $x_i y_i$ has multiplicity
one,
\[
 d_J(y_i)=1+\alpha_i+k_i,\qquad
 d_J(x_i)=1+\beta_i+l_i.
\]
Hence $U_i=\alpha_i+k_i-2$ and $V_i=\beta_i+l_i-2$ are
nonnegative.

The weight of a directed path is the product of its arc weights.
It counts the choices of one original edge copy for each arc of
the path. Let $p_i$ be the sum of these products over all directed
$S$--$i$ paths. Decomposing paths according to their last arc gives
\[
                    p_i=\alpha_i+\sum_{j=1}^n w_{ji}p_j.
\]
Topological induction now gives
$p_i\ge\alpha_i+2k_i\ge2$: at an internal vertex with no
internal predecessor, $\alpha_i\ge2$, and the same conclusion
follows successively from the recurrence.

Every perfect matching of $J$ either uses one of the $\ell$ copies
of $ab$ or corresponds to a directed $S$--$T$ path together with
a choice of the edge copies represented by its arcs. Indeed, its
symmetric difference with the unique matching of $J-a-b$ consists
of an alternating $a$--$b$ path; any additional alternating cycle
would contradict uniqueness in $J-a-b$. Consequently, if $P$
denotes the total weight of all directed $S$--$T$ paths, then
\[
                         P=\sum_i p_i\beta_i,
                    \qquad m(J)=\ell+P.
\]
Summing the recurrence for $p_i$ gives the identity
\begin{equation}\label{eq:pathidentity}
 e=\sum_i p_iV_i+2\sum_i U_i+
                      \sum_{i\to j}w_{ij}(p_i-2).
\end{equation}
For completeness, summing that recurrence first yields
$\sum_i p_i=\sum_i\alpha_i+\sum_i p_i l_i$. Add this zero
difference to $P=\sum_i p_i\beta_i$ to obtain
$P=\sum_i\alpha_i+\sum_i p_i(\beta_i+l_i-1)$.
Using $\beta_i+l_i-1=V_i+1$ and subtracting $4n$ gives
$e=\sum_i p_iV_i+\sum_i(p_i+\alpha_i-4)$.
Finally the recurrence and $U_j=\alpha_j+k_j-2$ give
$p_j+\alpha_j-4=2U_j+\sum_i w_{ij}(p_i-2)$;
substitution proves~\eqref{eq:pathidentity}.
Each summand there is nonnegative, since
$U_i,V_i,w_{ij}\ge0$ and $p_i\ge2$. More specifically,
\[
 e\ge2\sum_i V_i+2\sum_i U_i
   =2\sum_{v\notin\{a,b\}}(d_J(v)-3).
\]
This proves both $e\ge0$ and~\eqref{eq:nonhub}; the factor two
in front of $\sum V_i$ uses $p_i\ge2$, not merely $p_i\ge0$.

For the bound on individual matching edges, put
\[
 T_j=p_j+\alpha_j-4
     =2U_j+\sum_iw_{ij}(p_i-2)\ge0.
\]
The same identity reads $e=\sum_i p_iV_i+\sum_j T_j$.
Consequently $T_j\le e$, and $\alpha_j\ge0$ gives $p_j\le e+4$.
Let $r_i$ be the sum of the path weights over all directed
$i$--$T$ paths. Reversing every arc exchanges $\alpha$ with
$\beta$ and $U$ with $V$, without changing $e$, so the same
argument gives $r_i\le e+4$.
An $S$--$i$ path and an $i$--$T$ path have no other common
vertex, since otherwise their union contains a directed cycle.
They therefore concatenate uniquely, and multiplication of their
weights shows that the total weight of paths through $i$ is
$p_i r_i$. In the matching correspondence, a path through $i$
replaces the matching edge $x_i y_i$; a path avoiding $i$ leaves
it in place. The $\ell$ matchings using $ab$ also contain that
edge. Hence exactly $p_i r_i\le(e+4)^2$ matchings omit it.

To prove~\eqref{eq:hubdegree}, observe that
$p_i\ge\alpha_i+2k_i=2+U_i+k_i$. Thus
the last sum in~\eqref{eq:pathidentity} is
$\sum_i l_i(p_i-2)\ge\sum_i k_i l_i$, while its first sum
is at least $2\sum_iV_i$. It follows that
\[
             e\ge2\sum_i(U_i+V_i)+\sum_i k_i l_i.
\]
For nonnegative integers $k,l,U,V$ with $k\le2+U$ and $l\le2+V$,
\[
                         k+l-2\le kl/2+U+V.
\]
It suffices to take $U=(k-2)_+$ and $V=(l-2)_+$; the four cases
$k,l\le2$, $k\le2\le l$, $l\le2\le k$, and $k,l\ge2$ verify it.
Writing $E_0=\sum_i k_i=\sum_i l_i$ for the total internal arc
weight and summing this scalar inequality gives
\[
 2E_0-2n\le\frac12\sum_i k_i l_i+\sum_i(U_i+V_i)\le e/2.
\]
Thus $E_0\le n+e/4$. Since $d_J(a)=\ell+\sum_i\alpha_i$
and $\alpha_i=2+U_i-k_i$, we obtain
\[
 d_J(a)=\ell+2n+\sum U_i-E_0\ge n+\ell-e/4,
\]
For $b$, use $d_J(b)=\ell+\sum_i\beta_i$ and
$\beta_i=2+V_i-l_i$. Finally $e=B-\ell$ implies
$n+\ell-e/4=n-B/4+5\ell/4\ge n-B/4$. An integer degree
at least $n-B/4$ is at least $n-\lfloor B/4\rfloor$, proving
\eqref{eq:hubdegree}.
\end{proof}

\begin{lemma}\label{lem:dichotomy}
Let $B\ge0$ be an integer and $Q(B)=(B+5)^2$. If $J\in\cF$ has
rank $q\ge Q(B)$ and $m(J)\le4q-4+B$, then either $J$ has a unit
pair or
\begin{equation}\label{eq:dichotomy}
                       |E(J)|-3q\le3(B+5)^2+2B.
\end{equation}
\end{lemma}
\begin{proof}
We describe selected reductions of a nonsimple member of $\cF$.
If $\mu(J)>0$, choose a largest solitary class $ab$ and a parallel
class $az$ with $z\ne b$. Copies in the latter class are
nonsolitary by Lemma~\ref{lem:solitary}. If both its endpoints are noncubic, delete a copy;
otherwise contract at its cubic endpoint as in~\eqref{eq:extension}.
If $\mu(J)=0$, choose any parallel class and make the same decision.

For a selected deletion $J\to K$, its matching contribution is at
least two and $\mu(K)\ge\mu(J)$. Indeed, when a solitary class
$ab$ was selected, its unique cofactor matching cannot use an
edge incident with $a$; deleting a copy in $az$ leaves that
cofactor and the class $ab$ intact. If $\mu(J)=0$, monotonicity
is automatic. For a selected contraction, use
the notation $u,v,w,C_v,C_w$ of~\eqref{eq:extension}, and let
$s_v,s_w$ count solitary copies of $K$ in the two cut parts. Then
\begin{equation}\label{eq:selected}
 \mu(K)-\mu(J)\ge s_w,\qquad
 \sum_{f\in C_w}m_K(f)+\mu(K)-\mu(J)\ge2|C_w|\ge4.
\end{equation}
Indeed, adjacent solitary edges are parallel, so
$\mu(K)\ge s_v+s_w$. If the selected cubic endpoint is $z$, all
copies of the chosen $ab$ class lie in $C_v$ and remain solitary.
If it is $a$, then $\mu(J)=1$ and $uw=ab$; solitude implies that
$C_v$ consists of one solitary copy. Thus $s_v\ge\mu(J)$ in both
cases. If $\mu(J)=0$, the assertion is immediate.
Each nonsolitary cut copy belongs to at least two matchings and
each solitary copy to one. Thus
$\sum_{f\in C_w}m_K(f)\ge2|C_w|-s_w$; combining this with
$\mu(K)-\mu(J)\ge s_w$ proves the second inequality
in~\eqref{eq:selected}.

Set $\epsilon(J)=m(J)-4q+4-\mu(J)$. At ranks at least seven this
is nonnegative by~\eqref{eq:clmmu}. It decreases by at least two
under a selected deletion, and is nonincreasing under a selected
contraction. To track the rank change explicitly, put
$I=m(J)-m(K)$ and $\Delta\mu=\mu(K)-\mu(J)\ge0$. Then
\[
 \epsilon(J)-\epsilon(K)=
 \begin{cases}I+\Delta\mu,&\text{for deletion},\\
 I-4+\Delta\mu,&\text{for contraction}.
 \end{cases}
\]
The second expression is nonnegative by~\eqref{eq:selected}.
Moreover, a step with excess decrease at most $B$ has
matching contribution (the difference $m(J)-m(K)$) at most $B$ for deletion and at most $B+4$
for contraction, since $\mu(K)\ge\mu(J)$.

We next show that a unit pair propagates backwards through any such
step whose smaller graph has rank at least $Q(B)$ and excess at most
$B$. Choose its unit pair $(a,b)$ so that the multiplicity of $ab$
is $\mu(K)$: use a largest solitary class if $\mu(K)>0$, and a
unit nonedge otherwise. Lemma~\ref{lem:unitidentity} shows that
its parameter $e$ is exactly $\epsilon(K)\le B$. Since
$m(K)\ge4\operatorname{rank}(K)-4$,
each edge in the unique matching $M$ of $K-a-b$ belongs to at least
\[
                   4\operatorname{rank}(K)-4-(B+4)^2>B+4
\]
matchings. Here the strict inequality follows by substituting
$\operatorname{rank}(K)\ge(B+5)^2$; the resulting difference is
$3B^2+31B+76>0$. Call these edges protected for this step only.

For deletion, the restored parallel copy cannot be parallel to a
protected edge. It is not an $ab$ copy either, since it is
nonsolitary. A new matching after restoration would, on replacing
the restored copy by an old parallel copy, give $M$. That would
require the restored copy to be parallel to an edge of $M$, already
excluded. The unit pair persists and its
multiplicity, together with $\mu(K)\ge\mu(J)$, forces equality of
the maximum solitary multiplicities.

For contraction, suppose $u$ is on the $a$ side, and let $z$ be the
merged vertex corresponding to $v,w$ in $K$; interchange the shores
if necessary. No protected edge lies in $C_w$, since the entire
matching increment is at most $B+4$. If $z\ne b$, the edge of $M$
at $z$ lies in $C_v$ and has a unique extension. Every matching of
$J-a-b$ contracts to $M$: the vertex $u$ must be matched inside
$\{u,v,w\}$, leaving exactly one cut edge, just as in the tight-cut
argument for~\eqref{eq:extension}. Its edge at $z$ is in $C_v$,
so $uw$ is forced in the extension. This is again a unit pair. Its
multiplicity is unchanged and equals $\mu(J)$ by monotonicity.
If $z=b$, delete $(a,v)$ in $J$. The edge $uw$ is forced, and the
remaining matching is precisely $M$. All solitary cut copies at
$z$ are parallel to $a$, so the multiplicity of $av$ is
$\mu(K)-s_w$. Equation~\eqref{eq:selected} and the new unit pair
give $\mu(J)=\mu(K)-s_w$, including the case of multiplicity zero.
This proves backward propagation.

Now suppose $J$ has no unit pair. Run selected reductions down to rank
$Q(B)$. No intermediate graph at or above that rank can have a unit
pair: if one first appeared, backward propagation through the
preceding step, whose excess decrease is at most $B$, would give
one in its predecessor. Hence all have $\mu=0$ and excess between
zero and $B$. A simple terminal graph of larger rank is impossible,
because~\eqref{eq:simplebounds} would give $q\le(B+20)/2<Q(B)$.
There are at most $\lfloor B/2\rfloor$ deletions. At terminal rank,
\eqref{eq:dimension} and $m\le4q-4+B$ give
$E\le6q-6+B$, hence $E-3q\le3Q(B)-6+B$. Contractions
preserve $E-3q$, and a deletion decreases it by one. Restoring the
deleted copies therefore adds at most $\lfloor B/2\rfloor$ to
this last bound, which is at most the right side of
\eqref{eq:dichotomy}. The process cannot stop earlier: a nonsimple
graph admits a selected step, and each step reduces its size.
\end{proof}

\begin{theorem}\label{thm:gurvits}
\textup{\cite[Theorem~2.2]{Gurvits}}
Let $C$ be a nonnegative square matrix with positive permanent,
and let $Z$ be a nonnegative matrix with every row and column sum
equal to one, supported on the positive entries of $C$. Then
\[
 \log\per C\ge\sum_{i,j}\left[(1-Z_{ij})\log(1-Z_{ij})
                   -Z_{ij}\log\frac{Z_{ij}}{C_{ij}}\right],
\]
where zero terms are interpreted by continuity, including a zero
contribution when $C_{ij}=Z_{ij}=0$.
\end{theorem}
Here $\per C=\sum_{\pi}\prod_i C_{i,\pi(i)}$, with the sum over
permutations; for a bipartite multigraph its multiplicity matrix
has permanent equal to its number of perfect matchings.

\begin{lemma}\label{lem:entropy}
Let $G$ be a simple bipartite matching covered graph of rank $r$ and
minimum degree at least three. Put $S=|E(G)|-3r$. Then
\begin{equation}\label{eq:entropybound}
                         m(G)\ge(4/3)^{3r/(S+3)}.
\end{equation}
\end{lemma}
\begin{proof}
Set $L=S+3$ and $\eta=1/L$. Every degree is at most $L$. There
are nonnegative edge weights $y_e$ with vertex sums $L-d(v)$.
Indeed, for $\varnothing\ne X\subsetneq A$, put $Y=N(X)$ and
$t=|Y|-|X|\ge1$ by Theorem~\ref{thm:clmstructure}. Then
\[
 \sum_{y\in Y}d(y)-\sum_{x\in X}d(x)\le3t+S\le Lt,
\]
which is the required neighbourhood inequality for these demands.
The empty and full subsets satisfy the demand inequalities as well.
Here the degree excess on each shore is $S$, which proves both
$d(v)\le L$ and the first displayed inequality. Rewriting it as
\[
 \sum_{x\in X}(L-d(x))\le\sum_{y\in N(X)}(L-d(y))
\]
makes the demand condition explicit; the total demand on each
shore is $rL-|E(G)|$. To obtain $y_e$, replace each vertex $v$
by $L-d(v)$ copies, joining all copies of adjacent vertices.
The displayed inequality implies Hall's condition for the copy
graph: a set of copies with original vertices $X$ has size at
most the total demand of $X$ and sees all copies of $N(X)$.
Theorem~\ref{thm:hall} gives a perfect matching; let $y_e$ count
its edges over each original edge $e$.
Thus $x_e=(1+y_e)/L$ is a fractional perfect matching positive by
at least $\eta$ on every edge.

Theorem~\ref{thm:gurvits}, applied to the zero-one biadjacency matrix
and the matrix with entries $x_e$ (zero off the edges), gives
\[
 \log m(G)\ge\sum_e f(x_e),\qquad
 f(t)=-t\log t+(1-t)\log(1-t).
\]
For a probability vector with at least three positive coordinates,
all at least $\eta$, we have
\[
                    \sum_i f(p_i)\ge2f(\eta)-f(2\eta)
                                     \ge3\eta\log(4/3).
\]
Here a probability vector has nonnegative entries summing to one.
For the first inequality, transferring mass from a smaller entry
$b$ to a larger entry $a$ cannot increase the sum, because, for
$a\ge b>0$ with $a+b\le1$,
$f'(a)-f'(b)=\log[b(1-b)/(a(1-a))]\le0$.
Retain the second and third largest entries at least $\eta$,
transfer all smaller entries to the largest, and then reduce the
two retained entries to $\eta$ by the same transfers. The result
is $(1-2\eta,\eta,\eta,0,\ldots)$; continuity at zero and
$f(1-t)=-f(t)$ give $2f(\eta)-f(2\eta)$. For the second
inequality, $g(t)=2f(t)-f(2t)$ has
$g(0)=0$, $g(1/3)=\log(4/3)$, and
$g''(t)=-2/[(1-t)(1-2t)]<0$. Use the chord on $[0,1/3]$.
Each row has at least three positive entries because $G$ is simple
of minimum degree three. Summing its bound over the $r$ rows and
exponentiating proves~\eqref{eq:entropybound}.
\end{proof}

\section{Opposite deletions and common forests}\label{sec:forests}

We first dispose of opposite deletions with small edge excess. This
will later supply a lower bound on every column cofactor in an exact
forest expansion.

\begin{lemma}\label{lem:smallcore}
Let $G[A,B]$ be a simple brace of rank $r\ge3$, and let $a\in A$,
$b\in B$. Put $N=r-1$, $H=G-a-b$, $h=m(H)$ and
$\xi=|E(H)|-|V(H)|$. If $\xi=0$, then $G$ is a biwheel or an
extended biwheel. If $\xi\ge1$, then
\begin{equation}\label{eq:smallcore}
             m(G)\ge\frac76(N-4\xi)_+^2-\frac{\xi N}{2}.
\end{equation}
In particular, $1\le\xi\le N/59$ implies $m(G)>N^2$. The same
conclusion holds if $3\le h\le N/59+2$.
\end{lemma}
\begin{proof}
The graph $H$ is connected and matching covered. Such a graph of
order at least four has minimum degree at least two and has no
cutvertex. A leaf would force its incident edge and make any
other edge at its neighbour inadmissible. If $v$ were a cutvertex,
the existence of a perfect matching would force exactly one
component of $H-v$ to have odd order; every other component would
have even order. No edge from $v$ to an even component could
belong to a perfect matching, contrary to matching coverage.
Thus $\xi\ge0$, and $h\ge\xi+2$ by~\eqref{eq:dimension}. If
$\xi=0$, $H$ is a cycle. Every vertex of $H$ has its forced edge
to the appropriate hub in $G$, leaving only the hub edge optional.
This is the first assertion (including the extended biwheel $K_{3,3}$
at rank three).

Suppose $\xi\ge1$. Let $W$ be the vertices of $H$ of degree at
least three. Then $|W|\le2\xi$, and the edges of $H$ partition
into at most $|W|+\xi\le3\xi$ maximal paths with internal degree
two and endpoints in $W$. No path has equal endpoints, since $H$
has no cutvertex. Shorten each odd path of length at least three
to length three, each even path to length two, and leave direct
edges unchanged. The resulting graph $H_0$ is simple, matching
covered, and has the same $h$ perfect matchings. Odd subdivision
has a unique matching extension in each of its two endpoint states.
Thus shortening a path by an even number of internal vertices
preserves its matching choices. Paths that are not direct edges
retain internal vertices, so distinct paths do not become parallel
edges in $H_0$.
Also $|V(H_0)|\le |W|+2(3\xi)\le8\xi$, since each
shortened path has at most two internal vertices.

Retain $a,b$ and their edges to retained vertices, obtaining $G_0$.
We do not require $G_0$ to be a brace. To recover $G$, subdivide
distinct edges $u_i v_i$ of $H_0$ by $2k_i$ vertices, where
$k_i\ge1$, and add the forced hub edges at these vertices. With
$K=\sum_{i=1}^{\ell}k_i$,
\begin{equation}\label{eq:coreparameters}
               \ell\le3\xi,\qquad (N-4\xi)_+\le K\le N.
\end{equation}
Here $2K=2N-|V(H_0)|$; this gives the lower bound for $K$,
and the number of shortened paths gives the bound for $\ell$.
The case $K=\ell=0$ is permitted.

Write $u_i$ on the $a$ side and put
\begin{align*}
 z_i&=m(H_0-u_i-v_i),\\
 \Delta_i&=m(G_0-a-v_i)+m(G_0-u_i-b)+h-z_i,\\
 c_{ij}&=m(H_0-u_j-v_i)+m(H_0-u_i-v_j).
\end{align*}
Partition matchings by the new vertices used by the hubs. This gives
\begin{equation}\label{eq:refinement}
 m(G)=m(G_0)+\sum_i k_i\Delta_i
                +h\sum_i\binom{k_i}{2}+\sum_{i<j}c_{ij}k_i k_j.
\end{equation}
Here is the count for the only nonimmediate term. On a subdivided
path $u_i-y_1-x_1-\cdots-y_{k_i}-x_{k_i}-v_i$, if the hubs use
$y_j,x_t$, the remaining count is $h-z_i$ for $j\le t$ and $z_i$
for $j>t$. Their sum is $h\binom{k_i}{2}+k_i(h-z_i)$. If only
one hub uses this path the two linear cofactor terms result; if
they use distinct paths the two cross cofactors result. The same
partition applies when selected edges share endpoints.

Since $H_0$ is matching covered, every opposite deletion has a
perfect matching by Theorem~\ref{thm:clmstructure}. Moreover, $z_i$ counts
exactly the matchings of $H_0$ containing $u_iv_i$. That edge
cannot belong to every perfect matching, because its endpoint
has another incident edge and that other edge is admissible.
Hence $c_{ij}\ge2$ and $\Delta_i\ge h-z_i\ge1$.
In~\eqref{eq:refinement}, discard $m(G_0)\ge0$, replace
$\Delta_i$ by $1$ and $c_{ij}$ by $2$, and use
$2\sum_{i<j}k_i k_j=K^2-\sum_i k_i^2$. This yields
\[
 m(G)\ge K^2+\frac{h-2}{2}\Big(\sum_i k_i^2-K\Big)
       \ge\frac76K^2-\frac{\xi K}{2}.
\]
For the second inequality use $h-2\ge\xi$, the nonnegativity of
$\sum k_i(k_i-1)$, Cauchy--Schwarz, and $\ell\le3\xi$.
More explicitly, for $K>0$ the remaining lower bound is
\[
 K^2+\frac{\xi}{2}\Big(\sum_i k_i^2-K\Big)
 \ge K^2+\frac{\xi}{2}\Big(\frac{K^2}{3\xi}-K\Big).
\]
It is immediate if $K=0$. Use~\eqref{eq:coreparameters} separately
on the positive and negative terms to prove~\eqref{eq:smallcore}.
When $\xi\le N/59$, its right side is at least
\[
       \frac76N^2-\frac{59}{6}N\xi+\frac{56}{3}\xi^2
                      \ge N^2+\frac{56}{3}\xi^2>N^2.
\]
Finally $\xi\le h-2$, and $h\ge3$ excludes the cycle case.
\end{proof}

Fix a vertex $a\in A$ of degree $d$ in a simple brace $G[A,B]$ of
rank $r$ and size $s$. Let $C$ be its set of cubic neighbours,
$c=|C|$, and $t=d-c$. Throughout the forest arguments assume
$c\le r-2$. On $A\setminus\{a\}$ replace each vertex $b\in C$
by an edge $e_b$ joining its two other neighbours. The resulting
graph $F$, including isolated vertices, is a forest. Indeed, a
cycle with $j\le c\le r-2$ edges would give $j$ vertices of $C$
with only $j+1$ neighbours in $G$, contrary to~\eqref{eq:bracehall}.
Write
\begin{equation}\label{eq:forestparameters}
 N=r-1,\quad p=N-c,\quad
 \zeta=\sum_{x\in A\setminus\{a\}}(d_G(x)-3)=s-3r-d+3.
\end{equation}

For $b\in C$, first bicontract the $c-1$ degree-two vertices
represented by $F-e_b$ in $G-a-b$. A component $T$ becomes an
$A$ vertex of degree
\begin{equation}\label{eq:componentdegree}
       2+\sum_{x\in T}(d_G(x)-2)-|T\cap V(e_b)|.
\end{equation}
Indeed, its original degree sum is reduced by $2(|T|-1)$ for
the contracted tree edges and by one for each endpoint of $e_b$
in $T$. This gives the displayed formula, with all parallel
copies retained.
The two endpoints of $e_b$ lie in different components. Thus a
remaining degree-two $A$ vertex is a singleton that was a cubic
leaf of $F$: every summand $d_G(x)-2$ is at least one and at
most one endpoint of $e_b$ is counted in the subtracted term.
There are at most two such singletons.
Every remaining $B$ vertex either was not adjacent to $a$, or
had degree at least four before losing its edge to $a$. Hence
it has degree at least three. The final $j_b\le2$ bicontractions
identify such vertices and create degree at least four, so no
new degree-two vertex appears. Denote the resulting retract by
$J_b$ and its rank by $q_b$. We have
\begin{equation}\label{eq:exactretract}
 q_b=r-c-j_b,\qquad
 |E(J_b)|=3r-2c-2j_b-3+\zeta,
 \qquad |E(J_b)|-3q_b=c+j_b-3+\zeta.
\end{equation}
These identities subtract $c-1+j_b$ from the initial rank $r-1$
and twice that number from the initial size $s-d-2$; then use
$s=3r+d-3+\zeta$ from~\eqref{eq:forestparameters}.

Let $I$ and $L$ be, respectively, the sets of isolated vertices and
leaves of $F$ that are cubic in $G$, using the same letters for
their cardinalities. For $y\in B\setminus C$, let $i_y,l_y$
count its neighbours in these sets. The final contractions for
$J_b$ are represented by a forest $P_b$ on $B\setminus C$ with
at most two edges, one for each contracted leaf. A two-edge cycle
would force both corresponding $B$ vertices to be matched to
its two degree-two centres, contradicting matching coverage.
Let $N_b$ be the number of vertices in $I$ with at least two
neighbours in one component of $P_b$. By~\eqref{eq:componentdegree},
apart from the at most two components containing endpoints of $e_b$,
the cubic $A$ vertices are precisely the original cubic isolates:
a component of order at least two containing neither endpoint
has degree at least four. A surviving
isolate acquires parallel edges only if two of its original
neighbours are identified in a component of $P_b$. Therefore
\begin{equation}\label{eq:receiverpotential}
                         b_A(J_b)\le N_b+2.
\end{equation}

\begin{lemma}\label{lem:receivers}
With this notation,
\begin{equation}\label{eq:receiverbounds}
 \sum_{b\in C}N_b\le\frac12\sum_y l_yi_y
                 \le\frac18\sum_{y\in B}d_G(y)^2,
 \qquad \sum_{b\in C}N_b\le IL.
\end{equation}
For any real $M\ge3$, call a vertex of $B$ high if its degree is
at least $M$, and let $h_M$ count these vertices. At most
$(s-3r+3h_M)/2$ members $b\in C$ have a component of $P_b$
containing two high vertices. Every other member satisfies
$N_b\le2M$, and $h_M\le s/M$.
\end{lemma}
\begin{proof}
In a component of $P_b$ choose a vertex with largest $i_y$ as its
representative. Each counted isolate uses an edge to a
nonrepresentative, because it has at least two distinct neighbours
in the component. Select one such edge per counted isolate;
different isolates select different edges. Thus the number counted
in that component is at most $\sum i_y-\max i_y$. For a single
edge this is at most half the sum of the two endpoint values $i_y$.
For a two-edge path
$x$--$y$--$z$ it is at most $(i_x+2i_y+i_z)/2$, since
$i_x+i_z\le2\max\{i_x,i_y,i_z\}$. Each cubic leaf generates a
contraction only for its unique incident edge of $F$. Thus every
generating edge is used for only one $b$, and its endpoint
contribution is bounded by half the sum of their $i_y$ values.
The total number of such incidences at $y$ is $l_y$. Summing proves
the first bound. The sets $I,L$ are disjoint, so
$i_yl_y\le d_G(y)^2/4$. Also $N_b\le I$, and at most $L$ different
members can have a nonzero $N_b$, which proves the second bound.

A component with two high vertices consumes at least two high
endpoint incidences among its generating edges. These incidences
are not reused for another $b$. Their total is at most
$\sum_{y\text{ high}}d_G(y)\le s-3r+3h_M$. Otherwise choose the
single high vertex, if any, as representative. There are at most
two nonrepresentatives, each of degree less than $M$. This proves
the remaining assertions.
\end{proof}

\begin{lemma}\label{lem:aggregate}
Put
\[
 \cB=3rd-d^2+(s-3r-2d)c,\qquad K_0=(s-4r-d+2c)_+.
\]
Then
\begin{align}
 m(G)&\ge\cB-\sum_{b\in C}N_b-12d-4s,\label{eq:aggregate}\\
 m(G)&\ge\cB-\frac r8(s-3r)-12d-\frac{35}{8}s.
                                                        \label{eq:squares}
\end{align}
If $s\le4r$, then for any $M\ge3$,
\begin{equation}\label{eq:tunable}
 m(G)\ge\cB-\frac12(s-3r)K_0-2Mc-\frac{6rK_0}{M}-28r.
\end{equation}
If $3r/4\le d\le r-1$ and $s\le5r$, then
\begin{equation}\label{eq:highforest}
 m(G)\ge d(5r/2-3d/2)+(d-c)r/3-35r.
\end{equation}
\end{lemma}
\begin{proof}
For $b\in C$, Lemma~\ref{lem:potential},
\eqref{eq:receiverpotential} and~\eqref{eq:clmlinear} give the
convenient weakened estimate
\begin{equation}\label{eq:cubiccofactor}
 m(G-a-b)\ge\max\{s-d-2c-22-N_b,\ 4r-4c-16\}.
\end{equation}
For the first alternative, use
$|E(J_b)|=s-d-2c-2j_b$, subtract at most $N_b+2+14$, and
then use $j_b\le2$. For the second, use $q_b\ge r-c-2$ in
$m(J_b)\ge4q_b-8$. Both alternatives estimate the same cofactor,
so their maximum is also a lower bound.
For a noncubic neighbour $b$ of degree $k$, there are initially
$c$ degree-two $B$ vertices and at most $k-1$ degree-two $A$
vertices in $G-a-b$. Its retract has at least
$s-d-2c-3k+3$ edges, since at most $c+k-1$ bicontractions
remove two edges each from the initial $s-d-k+1$ edges.
To apply Lemma~\ref{lem:fibercharge}, note that each surviving
$B$ vertex $y$ satisfies
$(d_{G-a-b}(y)-3)_+\le d_G(y)-3$. Enlarge its nonnegative
sum to include all $y\in B\setminus\{b\}$, whose original
degree excess totals $s-k-3(r-1)$. The lemma then bounds the
number of cubic parallel $A$ vertices by
$2(k-1)+s-k-3(r-1)=s-3r+k+1$. Hence
\begin{equation}\label{eq:noncubiccofactor}
                  m(G-a-b)\ge3r-d-2c-4k-12.
\end{equation}
Terminal cycles satisfy these numerical inequalities directly.
Sum~\eqref{eq:cubiccofactor} and~\eqref{eq:noncubiccofactor}
at $a$. Before simplifying, the resulting bound is
\[
 \begin{split}
 m(G)\ge{}&c(s-d-2c-22)-\sum_{b\in C}N_b\\
 &+(d-c)(3r-d-2c-12)
       -4\sum_{b\in N(a)\setminus C}d(b).
 \end{split}
\]
The last degree sum is at most $s-3c$, since the omitted $c$
vertices are cubic. Multiplying this upper bound by $-4$ preserves
the required lower bound.
The resulting leading term is
$\cB=d(s-3d)+(d-c)(2d-s+3r)$; discarding a nonnegative $2c$
gives~\eqref{eq:aggregate}.

Since $3\le d_G(y)\le r$, write
$d_G(y)^2=3d_G(y)+d_G(y)(d_G(y)-3)$ and replace the first
factor in the last product by $r$. Summing gives
$\sum_y d_G(y)^2\le r(s-3r)+3s$.
Substitute this and~\eqref{eq:receiverbounds} into
\eqref{eq:aggregate} to obtain~\eqref{eq:squares}.
For~\eqref{eq:tunable}, retain the second alternative
in~\eqref{eq:cubiccofactor} when $P_b$ contains two high vertices.
Its loss from $s-d-2c-22$ is at most $K_0$, since the difference
between that expression and the second alternative is
$s-4r-d+2c-6\le K_0$. On the other members the first alternative
has loss $N_b\le2M$. Charging the latter bound to all $c$ members
and the former bound additionally to exceptional members only
weakens the lower bound. The number of exceptional members is at most
$(s-3r)/2+6r/M$, and $12d+4s\le28r$. This proves
\eqref{eq:tunable}; it does not require $d\ge r/2$.

It remains to prove~\eqref{eq:highforest}. The tree degree sum
implies
\begin{equation}\label{eq:leaves}
                  2L\le N+2p-3I+\zeta,\qquad I+L\le N.
\end{equation}
Indeed, if $J$ counts all isolates, the number of leaves in the
nontrivial components is $2(p-J)+\sum_{d_F(x)\ge3}(d_F(x)-2)$.
The last sum is at most the number of nonleaf, nonisolated vertices
plus $\zeta$. Use $J\ge I$ and then $L$ no larger than the total
number of leaves to obtain~\eqref{eq:leaves}.

Set $x=d/r$, $z=c/r$, $i=I/r$, and
$\widetilde u=(s-3r+3)/r\ge x$. Equations
\eqref{eq:aggregate}, \eqref{eq:receiverbounds} and
\eqref{eq:leaves} give
\[
 \frac{m(G)}{r^2}\ge3x-x^2+(\widetilde u-2x)z-iL/r-35/r,
\]
where $0\le i\le1-z$, $iL/r\le1/4$, and
$iL/r\le i(3-2z-3i+\widetilde u-x)/2$.
For $z\ge1/3$, substitute the latter bound. Its derivative in
$\widetilde u$ is $z-i/2\ge0$, so set $\widetilde u=x$.
The maximum of $i(3-2z-3i)/2$ for $0\le i\le1-z$ is
\[
 Q_*(z)=
 \begin{cases}(3-2z)^2/24,&z\le3/4,\\
               z(1-z)/2,&z\ge3/4.
 \end{cases}
\]
The derivatives of $3x-x^2-xz-Q_*(z)$ on its two pieces are
$-x+1/2-z/3$ and $-x-1/2+z$, both at most $-1/3$ for
$3/4\le x\le1$ and $1/3\le z\le x$. Its value is consequently
at least its value at $z=x$ plus $(x-z)/3$. At $z=x\ge3/4$
that value is $x(5/2-3x/2)$, giving the asserted bound.
For $z\le1/3$, use $iL/r\le1/4$ instead. The difference from the
same target is
\[
 x/6+x^2/2-1/4+z(1/3-x)
       \ge x^2/2-x/6-5/36\ge5/288>0.
\]
The finite error is at most $12d+4s+3c\le35r$, proving
\eqref{eq:highforest}.
\end{proof}

\begin{lemma}\label{lem:forestidentity}
Let $T_1,\ldots,T_p$ be the components of $F$ and $Q=B\setminus C$,
so $|Q|=p+1$. Define
\[
 S_{iy}=N_G(y)\cap V(T_i),\qquad D_{iy}=|S_{iy}|\quad(y\in Q).
\]
Let $h_y$ be the permanent of $D$ with column $y$ removed, and
$P_i(y,z)$ the permanent with row $i$ and columns $y,z$ removed.
The empty permanent is one. Then $h_y=m(G-a-y)$ and
\begin{equation}\label{eq:forestidentity}
 m(G)=\sum_{y\in Q\cap N(a)}h_y+
       \sum_i\sum_{\{y,z\}\subset Q}P_i(y,z)W_i(y,z),
\end{equation}
where $W_i(y,z)=\sum_{u\in S_{iy},v\in S_{iz}}\dist_{T_i}(u,v)$.
If $\ell\ge1$ is an integer and every radius-$(\ell-1)$ ball
in the trees has at most $b$ vertices,
then
\begin{equation}\label{eq:allstars}
              2m(G)\ge\ell\sum_{y\in Q}h_y
                                   \sum_{i=1}^p(D_{iy}-b)_+.
\end{equation}
\end{lemma}
\begin{proof}
In $G-a-y$ bicontract all the centres represented by $F$ to obtain
the matrix defining $h_y$. Each tree becomes one row vertex,
and its $D_{iz}$ edges to column $z$ are retained as distinct
parallel copies. Choosing one entry in each row and column
therefore counts exactly the perfect matchings of the retract.
For a matching using the cubic neighbour
corresponding to an edge of $T_i$, contract $F$ minus that edge.
The two split components use two distinct columns $y,z$, leaving
the factor $P_i(y,z)$ for all other components. A choice
$u\in S_{iy}$, $v\in S_{iz}$ is possible for this split exactly
when the removed tree edge separates $u$ and $v$. In a tree
these are exactly the edges on their unique path, numbering
$\dist_{T_i}(u,v)$. Together with matchings that pair $a$ to
a vertex of $Q$, this proves
\eqref{eq:forestidentity}.

For each $v\in S_{iz}$, at least $(D_{iy}-b)_+$ members of
$S_{iy}$ are at distance at least $\ell$. Thus
\[
 \sum_{z\ne y}P_i(y,z)W_i(y,z)
 \ge\ell(D_{iy}-b)_+\sum_{z\ne y}P_i(y,z)D_{iz}
 =\ell(D_{iy}-b)_+h_y.
\]
The last equality is expansion of the permanent along row $i$.
All permanents are nonnegative, so the preceding distance bound
can be multiplied by $P_i(y,z)$ and summed without changing its
direction.
Summing over all $i,y$ counts each unordered pair exactly twice
on the left. The first term in~\eqref{eq:forestidentity} is
nonnegative, which proves~\eqref{eq:allstars}.
\end{proof}

\begin{corollary}\label{cor:boundedforest}
Suppose $p\le P$, the trees have maximum degree at most $D\ge2$,
and $r-1\ge119P(P+1)D^{119}$. Then $m(G)\ge(r-1)^2$.
\end{corollary}
\begin{proof}
Suppose $m(G)<N^2$. Lemma~\ref{lem:smallcore} implies
$h_y>N/59+2$ for every $y\in Q$; a cofactor of two would already
give a biwheel or an extended biwheel. Indeed, $G-a-y$ has
minimum degree at least two and the dimension bound gives $h_y\ge2$;
if equality holds, $|E(G-a-y)|-|V(G-a-y)|=0$. The total number
of incidences counted by the entries $D_{iy}$ is
\[
                  \sum_{i,y}D_{iy}=N+2p+\zeta\ge N.
\]
This counts the edges from $A\setminus\{a\}$ to $Q$: from
the degree sum $3N+\zeta$, remove the $2c$ edges to $C$ and use
$c=N-p$.
A ball of radius $\ell-1$ in a tree of maximum degree $D$ has at
most $D^\ell$ vertices, also when $D=2$. Apply
\eqref{eq:allstars} with $\ell=119$ and $b=D^{119}$. Its right
side is bounded using
\[
 \sum_{y\in Q}\sum_{i=1}^p(D_{iy}-b)_+
 \ge\sum_{i,y}D_{iy}-p(p+1)b\ge N-P(P+1)b.
\]
The last quantity is positive by hypothesis. Since each
$h_y>N/59$, the right side of~\eqref{eq:allstars} is strictly
greater than
\[
           \frac{119N}{59}\{N-P(P+1)b\}\ge2N^2,
\]
a contradiction. The expression in braces is positive by the
hypothesis, justifying the strict comparison.
\end{proof}
\section{Excluding large degrees}\label{sec:degree}

In this section we prove the degree restriction needed in the final
reduction. Retain the constant $R$ in~\eqref{eq:threshold}.

\begin{theorem}\label{thm:degree}
Let $G$ be a simple brace of rank $r\ge R$. If
$m(G)<(r-1)^2$, then $\Delta(G)<r/2-4$.
\end{theorem}

We use the notation $d,c,t,\zeta,F$ of Section~\ref{sec:forests}
whenever a vertex $a$ is specified. All elementary estimates below
are applied at $r\ge R>10^8$.

\begin{lemma}\label{lem:nearfull}
A simple brace of rank $r\ge R$ with a vertex of degree at least
$3r/4$ satisfies $m(G)\ge(r-1)^2$.
\end{lemma}
\begin{proof}
Suppose otherwise, and write $d=r-k$. For each $b\in N(a)$ define
the nonnegative integer
\[
 \sigma_b=m(G-a-b)-(s-d-d_G(b)-2r+5),
\]
whose nonnegativity is~\eqref{eq:dimcofactor}, and put
$\tau=\sum_{b\notin N(a)}(d_G(b)-3)\ge0$. Partitioning
matchings at $a$ gives the exact identity
\begin{equation}\label{eq:exactdeficit}
 m(G)-(r-1)^2=2-k(r-2)+(r-k-1)\zeta+\tau+
                                      \sum_{b\in N(a)}\sigma_b.
\end{equation}
Indeed, substitute $s=4r-k-3+\zeta$ and
$\sum_{b\in N(a)}d_G(b)=s-3k-\tau$ into the partition.
For $k=0$ this already contradicts the assumption. If
$d=c=r-1$, the auxiliary graph has $r-1$ edges and $r-1$
vertices. Condition~\eqref{eq:bracehall} excludes every shorter cycle, so it
is a spanning cycle: it cannot be a forest with that many edges,
and its only permitted cycle already uses every vertex and edge.
Minimum degree forces every remaining spoke,
and $G$ is the biwheel. We may therefore assume $k\ge1$ and
$c\le r-2$.

If $s\ge5r$, \eqref{eq:dimensionsum} gives
\[
 m(G)\ge3rd-d^2+2d-2r\ge\frac{27}{16}r^2-2r>(r-1)^2.
\]
Thus $s<5r$, and~\eqref{eq:highforest} applies. Since
$d(5r/2-3d/2)-r^2=k(3d/2-r)\ge kr/8$, comparison with
$(r-1)^2$ gives $k/8+t/3<33+1/r$. Multiplication by $24$
and integrality (with $r>24$) yield
\begin{equation}\label{eq:nearfullparameters}
                3k+8t\le792,\qquad k\le264,\quad t\le99.
\end{equation}
If $\zeta\ge k+1$, \eqref{eq:exactdeficit} gives
$m(G)-(r-1)^2\ge r-k^2+1>0$: discard the nonnegative
$\tau$ and $\sigma_b$ terms and use $r-k-1>0$ to replace
$\zeta$ by $k+1$. Hence $\zeta\le k$.

For $b\in C$, bicontraction preserves both the matching count
and the expression $|E|-2\operatorname{rank}+2$, since each
bicontraction removes two edges and lowers the rank by one.
Thus $\sigma_b=m(J_b)-(|E(J_b)|-2q_b+2)$.
Lemma~\ref{lem:linearaux} gives
$\sigma_b\ge q_b-2\ge k+t-4$, also for a terminal double edge.
Sum this over the $c=r-k-t$ cubic neighbours and use only
$\sigma_b\ge0$ for the others. Substitution in
\eqref{eq:exactdeficit}, with $\tau\ge0$, yields
\begin{equation}\label{eq:nearfulldeficit}
 \begin{split}
 m(G)-(r-1)^2\ge{}&(t+\zeta-4)r-(k+t)(k+t-4)\\
                 &-(k+1)\zeta+2k+2.
 \end{split}
\end{equation}
If $t+\zeta\ge5$, its right side is at least $r$ minus
$2k^2+2kt+t^2-5k-4t-2$. This convex quadratic has maximum
$138070$ on the triangle $k,t\ge0$, $3k+8t\le792$:
its values at the three vertices are $-2$, $138070$ and $9403$.
Thus $t+\zeta\le4$. In particular,
\[
                 p=k+t-1\le263,\qquad \Delta(F)\le7.
\]
The first bound follows from $3k+8t\le792$. For the second,
each vertex of $F$ has degree at most its degree in $G$, which
is at most $3+\zeta\le7$ on $A\setminus\{a\}$.
Now apply Corollary~\ref{cor:boundedforest} with $P=263$ and
$D=7$. The definition of $R$ is precisely the required threshold.
\end{proof}

\begin{lemma}\label{lem:middle}
A simple brace of rank $r\ge R$ with a vertex of degree in
$[3r/5,3r/4]$ satisfies $m(G)\ge(r-1)^2$.
\end{lemma}
\begin{proof}
The forest exists in this range. If $s\ge5r$, use
\eqref{eq:dimensionsum} to obtain
$m(G)\ge(36/25)r^2-2r>(r-1)^2$.
Suppose $s\le3r+2d$. Then
$\cB\ge d(s-3d)$, because
$\cB=d(s-3d)+(d-c)(2d-s+3r)$ and both factors in the last
term are nonnegative. Consequently~\eqref{eq:squares} has positive
coefficient $d-r/8-35/8$ in $s$. Substituting
$s\ge3r+d-3$ gives
\begin{align*}
 m(G)&\ge d(23r/8-2d)-(155/8)d-(51/4)r+105/8\\
     &\ge(201/200)r^2-33r>(r-1)^2.
\end{align*}
Here $x(23/8-2x)$ is concave on $[3/5,3/4]$, with endpoint
values $201/200$ and $33/32$. The last strict comparison holds
already for $r\ge6201$.
If $3r+2d\le s\le5r$, use instead
$\cB=3rd-d^2+(s-3r-2d)c\ge3rd-d^2$, discarding its
nonnegative last term. Equation~\eqref{eq:squares} now gives
\[
              m(G)\ge(119/100)r^2-(271/8)r>(r-1)^2.
\]
This uses $3x-x^2-1/4\ge119/100$ on the same interval.
\end{proof}

\begin{lemma}\label{lem:window}
Suppose $r\ge R$, $m(G)<(r-1)^2$, and
$r/2-4\le d=d_G(a)<3r/5$. Put $e=d-r/2$. Then
\begin{equation}\label{eq:window}
             -4\le e<2134,\qquad \zeta\le12803,
             \qquad t\le3200.
\end{equation}
\end{lemma}
\begin{proof}
We first show that $s<39r/10$. For $d\ge r/2$, if $s\ge9r/2$,
\eqref{eq:dimensionsum} gives $m(G)\ge r^2-r/2>(r-1)^2$.
For $39r/10\le s\le9r/2$, put $x=d/r$, $z=c/r$, and
$u=(s-3r)/r$. The leading expression in~\eqref{eq:squares} is
\[
                     3x-x^2+(u-2x)z-u/8.
\]
Here $0\le z\le x$. If the coefficient $u-2x$ is negative,
the expression is minimized by setting $z=x$; if it is positive,
it is minimized by setting $z=0$.
If $u\le2x$ it is at least $3x-3x^2+u(x-1/8)\ge87/80$;
if $u\ge2x$ it is at least $3x-x^2-u/8\ge17/16$.
The first lower bound uses $u\ge9/10$, then minimizes the
resulting concave quadratic in $x$ at the endpoints of
$[1/2,3/5]$; the second uses $u\le3/2$ and $x\ge1/2$.
The error is at most $32r$, giving a contradiction.

For $r/2-4\le d<r/2$, if $s\ge23r/5$,
\eqref{eq:dimensionsum} gives
$m(G)\ge(21/20)r^2-7r-24>(r-1)^2$.
For $39r/10\le s\le23r/5$, moving $x$ up to $1/2$ in the
preceding two leading expressions loses at most $16/r$ and
$12/r$, respectively; their $x$ derivatives have absolute values
at most four and three. At $x=1/2$ the bounds are $87/80$ and
$21/20$, since $u\le8/5$. The error in~\eqref{eq:squares} is
less than $33r$. Thus $m(G)\ge(21/20)r^2-49r>(r-1)^2$.
All these strict comparisons hold for $r\ge1000$.

We may now assume $u<9/10$ and apply~\eqref{eq:tunable} with
$M=240$. Define
\[
 F(x,z,u)=3x-x^2+(u-2x)z
                        -\frac u2(u-1-x+2z)_+.
\]
Divide the definition of $K_0$ by $r$ to get
$K_0/r=(u-1-x+2z)_+$. Substituting this and $M=240$ into
\eqref{eq:tunable} gives
\begin{equation}\label{eq:Fbound}
                    m(G)\ge r^2F(x,z,u)-480c-rK_0/40-28r.
\end{equation}
For $1/2\le x\le3/5$, $x\le u_0\le9/10$, and $0\le z\le x$,
put $\alpha=x-1/2$ and $v=u_0-x$. Then
\begin{equation}\label{eq:scalar}
                 F(x,z,u_0)\ge1+\alpha/5+v/20+(x-z)/10.
\end{equation}
Indeed, the two slopes in $z$ are $u_0-2x$ and $-2x$, both at
most $-1/10$. At $z=x$ we have the exact expression
\[
 F(x,x,u_0)=x(7/2-3x)+(u_0-x)(1-u_0)/2.
\]
Now $x(7/2-3x)-1=\alpha/2-3\alpha^2\ge\alpha/5$ and
$1-u_0\ge1/10$, proving~\eqref{eq:scalar}.
The slope bound gives the additional $(x-z)/10$ by integrating
from $z$ to $x$; it holds on both linear pieces and hence also
across their common breakpoint.

Suppose first that $e\ge0$. Take $u_0=\max\{u,x\}$, so
$0\le u_0-u\le3/r$ and $rv=(\zeta-3)_+$. The absolute
$u$ derivative of $F$ is at most two, hence replacement costs
at most $6/r$. These assertions use
$\zeta=r(u-x)+3\ge0$; the derivative estimate applies on
each piece of the continuous, piecewise differentiable function.
Also $K_0/r\le2\alpha+v$ and $480c\le288r$.
In~\eqref{eq:Fbound}, substitute~\eqref{eq:scalar}, subtract
the replacement error $6r$, and bound
$rK_0/40\le r^2(2\alpha+v)/40$. Using
$r\alpha=e$, $rv=(\zeta-3)_+$ and $r(x-z)=t$ gives
\[
 m(G)\ge r^2+\frac{3r}{20}e
                    +\frac r{40}(\zeta-3)_++\frac r{10}t-322r.
\]
Comparing with $(r-1)^2$, and using integrality, gives
\begin{equation}\label{eq:upperwindow}
                          6e+(\zeta-3)_++4t\le12800.
\end{equation}
Although $e=d-r/2$ may be a half-integer, $6e=6d-3r$ is an
integer. The whole left side is therefore an integer, and the
comparison bounds it strictly below $12800+40/r<12801$.
This proves~\eqref{eq:window} in this case.

If $-4\le e<0$, use $u_0=\max\{u,1/2\}$ and
$v=u_0-1/2$. We have $0\le u_0-u\le7/r$,
$K_0/r\le v$, and $rv=(\zeta-3+e)_+$. The absolute partial
derivatives of $F$ in $x,u$ are at most five and two on this
domain. Changing $(x,u)$ to $(1/2,u_0)$ therefore costs at most
$34/r$. In~\eqref{eq:Fbound}, apply~\eqref{eq:scalar} at
$x=1/2$, subtract $34r$, and use $480c\le240r$ and
$rK_0/40\le r^2v/40$. Since $r(1/2-z)=t-e$, we obtain
\[
 m(G)\ge r^2+\frac r{40}(\zeta-3+e)_+
                       +\frac r{10}(t-e)-302r.
\]
Comparison and half-integrality give
\begin{equation}\label{eq:lowerwindow}
                      (\zeta-3+e)_++4(t-e)\le12000.
\end{equation}
Indeed the left side is a half-integer and is strictly less than
$12000+40/r$, where $r>80$. This implies $\zeta\le12007$
and $t\le3000$, which suffice for~\eqref{eq:window}.
\end{proof}

\begin{lemma}\label{lem:treecut}
Let $T$ be a tree of order $N$ and maximum degree at most $D\ge2$.
For an edge $f$, let $s_f$ be the smaller order of the two components
of $T-f$. If $\ell\ge1$ is an integer and $K=D^\ell$, then
\begin{equation}\label{eq:treecut}
                    \sum_{f\in E(T)}\min\{s_f,K\}
                                       \ge\ell(N-K).
\end{equation}
\end{lemma}
\begin{proof}
Root $T$ at a centroid. The order of the descendant subtree below
each edge is $s_f$: a centroid has no component of order greater
than $N/2$ after its deletion, and every descendant subtree lies
in one such component. Each vertex at depth at least $\ell$ contributes
one incidence to each of its last $\ell$ ancestral edges. An edge
receives incidences only from its descendant subtree, and only
from vertices at distance at most $\ell-1$ from its lower end.
It therefore receives at most $s_f$ incidences and at most
$D^\ell=K$, by the tree-ball bound. At most $K$ vertices have
depth less than $\ell$, since they lie in the corresponding ball
about the root. Counting by vertices gives at least
$\ell(N-K)$ incidences; counting by edges gives at most the
left side of~\eqref{eq:treecut}.
\end{proof}

\begin{proof}[Proof of Theorem~\ref{thm:degree}]
Lemmas~\ref{lem:nearfull} and~\ref{lem:middle} leave only
$r/2-4\le d<3r/5$. Suppose there is such a vertex in a strict
counterexample, and use Lemma~\ref{lem:window}.
Since $q_b\ge r-c-2\ge r/2-2136$, all cubic-neighbour retracts
have rank at least seven. Therefore
\begin{equation}\label{eq:strongcubic}
                      m(G-a-b)\ge4r-4c-12\quad(b\in C).
\end{equation}
For a noncubic neighbour $y$, \eqref{eq:dimcofactor} gives
$m(G-a-y)\ge r+\zeta+2-d_G(y)$, and
\[
                  \sum_{y\in N(a)\setminus C}d_G(y)
                                       \le d-3+\zeta+3t.
\]
Indeed, the other $r-t$ vertices on the $B$ shore have total
degree at least $3(r-t)$, and $s=3r+d-3+\zeta$.
Partitioning at $a$ gives the intermediate bound
$m(G)\ge c(4r-4c-12)+t(r+\zeta+2)-(d-3+\zeta+3t)$.
Subtract $r^2$ and substitute $d=r/2+e$, $c=r/2+e-t$ to obtain
\begin{equation}\label{eq:tpartition}
 \begin{split}
 m(G)-r^2\ge{}&(t-13/2)r-13e+12t-4(e-t)^2\\
              &+t(\zeta-1)+3-\zeta.
 \end{split}
\end{equation}
If $t\ge5$, the bounds~\eqref{eq:window} make this at least
$-3r/2-42000000>-2r+1$, a contradiction for $r\ge10^8$.
Here $|e-t|\le3204$, $-13e\ge-13\cdot2134$,
$t(\zeta-1)\ge-3200$, and $3-\zeta\ge-12800$ justify the
uniform error. Thus $t\le4$ and
\begin{equation}\label{eq:delta}
                         \delta=c-r/2=e-t\in[-8,2134).
\end{equation}

The matching partition gives
$\sum_{b\in C}m(J_b)\le m(G)<(r-1)^2$, since the other
cofactors are nonnegative. As $c>0$, some $b\in C$ therefore
has $m(J_b)<(r-1)^2/c$. The identity
\begin{equation}\label{eq:averagecofactor}
 \frac{(r-1)^2}{c}=4(r-c-2)+4+\frac{(2\delta+1)^2}{c}
\end{equation}
and $c\ge r/2-8$ show that the last fraction is less than one.
By integrality, this selected retract $J$ satisfies
$m(J)\le4(r-c-2)+4\le4q+4$, where the last inequality uses
$q=r-c-j_b\ge r-c-2$. The lower bound $m(J)\ge4q-4$ follows
from $q\ge7$. Also~\eqref{eq:exactretract} gives
$|E(J)|-3q\ge r/2-11$. Lemma~\ref{lem:dichotomy} with $B=8$
forces a unit pair, since $q\ge169$ and the size excess exceeds
$3\cdot169+16=523$. Lemma~\ref{lem:unitidentity} then bounds
the total degree excess outside the pair by four.

Every unaffected component $U$ of $F-e_b$ becomes an $A$ vertex
of degree $|U|+2+\zeta_U$; an affected one has degree
$|U|+1+\zeta_U$, where
$\zeta_U=\sum_{x\in U}(d_G(x)-3)\ge0$. Equivalently, if
$\varepsilon_U=|U\cap V(e_b)|\in\{0,1\}$ and $v_U$ is the
surviving vertex, then
\[
 |E(F[U])|=|U|-1
        =d_J(v_U)-3-\zeta_U+\varepsilon_U
        \le d_J(v_U)-3+\varepsilon_U.
\]
The later bicontractions identify only $B$ vertices, so these
$A$ degrees do not change. Removed $A$ centres are singletons
and contain no forest edges. Apart from
the component corresponding to the $A$ vertex of the unit pair,
the total number of forest edges in the surviving components is
at most the degree excess outside the unit pair plus
$\sum_U\varepsilon_U\le2$,
and hence at most six. Restoring $e_b$, all but at most seven
edges of $F$ lie in one component $T$. Its order satisfies
\begin{equation}\label{eq:largetree}
                 N_T\ge c-6\ge r/2-14,\qquad
                 \Delta(T)\le D:=12806.
\end{equation}

For every $b\in C$, define the raw cofactor excess
$B_b=m(J_b)-4q_b+4\ge0$. Partition at $a$ again and use
$q_b=r-c-j_b$ and $\sum_{b\in C}j_b\le2c$. Since
$\sum_{b\in C}m(J_b)<(r-1)^2$, this gives
\begin{equation}\label{eq:totalbudget}
 \begin{split}
 \sum_{b\in C}B_b
 &<(r-1)^2-4c(r-c)+12c\\
 &=4r+4\delta^2+12\delta+1\le5r.
 \end{split}
\end{equation}
The last constant term is at most $18241433$ on the interval
in~\eqref{eq:delta}.

Set
\begin{equation}\label{eq:treeconstants}
 K=D^8,\qquad B_0=2K,\qquad T_0=100(B_0+5)^2+10000.
\end{equation}
These constants are below $R$: since $D<2^{14}$, we have
$T_0<2^{235}<7^{100}<R$.
For $f\in E(T)$ corresponding to $b\in C$, let $s_f$ be its
smaller component order. We claim
\begin{equation}\label{eq:cutbudget}
                         B_b\ge2\min\{s_f,K\}-4.
\end{equation}
This is immediate if $B_b>B_0$. Otherwise
\eqref{eq:exactretract}, \eqref{eq:delta}, and $r\ge T_0$ give
\[
 q_b\ge r/2-2136\ge(B_0+5)^2,\qquad
 |E(J_b)|-3q_b\ge r/2-11>3(B_0+5)^2+2B_0.
\]
Lemma~\ref{lem:dichotomy}, applied with budget $B_0$, forces a
unit pair in $J_b$. By~\eqref{eq:nonhub}, the total degree excess
outside this pair is at most $B_b/2$: the parameter of that lemma is
$e=B_b-\ell\le B_b$, where $\ell$ is the multiplicity of the
unit pair. If $s_f\ge2$, both components of
$T-f$ survive as $A$ vertices, each with excess at least its
order minus two, by~\eqref{eq:componentdegree}. At most one can
be in the unit pair, which contains only one $A$ vertex. The
other component has order at least $s_f$, regardless of which
component is larger. Its degree excess alone is therefore at
least $s_f-2$. Thus $B_b\ge2s_f-4$; if $s_f=1$,
nonnegativity suffices. This proves~\eqref{eq:cutbudget}.

Sum~\eqref{eq:cutbudget} over $E(T)$ and apply
Lemma~\ref{lem:treecut} with $\ell=8$. Equations
\eqref{eq:largetree} and~\eqref{eq:totalbudget} give
\[
 5r\ge\sum_{b\in C}B_b
      \ge16(N_T-K)-4(N_T-1)
      \ge6r-16K-164.
\]
In the first lower bound we retain only the $B_b$ corresponding
to edges of $T$; all other $B_b$ are nonnegative. There are
$N_T-1$ edges in $T$, accounting for the term $-4(N_T-1)$.
But $r\ge T_0>16K+164$, a contradiction. This proves the theorem.
\end{proof}
\section{Strictly thin edges and the main theorem}\label{sec:completion}

The restriction to strictly thin edges is important here. Simplicity
of the retract of an edge deletion places useful restrictions on the
different retract arising from deletion of its two endpoints.

\begin{lemma}\label{lem:strictfibres}
Let $uv$ be a strictly thin edge of index two in a simple brace
$G$ of rank at least eight. Its opposite deletion $H=G-u-v$ can
be retracted using at most four bicontractions. When these are
performed by deleting a degree-two vertex and identifying its two
neighbours, the identifications involve pairwise disjoint pairs
of vertices of $H$. In particular, if $J=\ret H$, then
\begin{equation}\label{eq:strictmax}
                            \Delta(J)\le2\Delta(G)-2.
\end{equation}
\end{lemma}
\begin{proof}
Write $u\in A$, $v\in B$ and
$N(u)\setminus\{v\}=\{b_1,b_2\}$,
$N(v)\setminus\{u\}=\{a_1,a_2\}$. Bicontracting $u$ and $v$
in $G-uv$ merges these two pairs. Simplicity of that retract implies
\begin{equation}\label{eq:strictconditions}
 \begin{gathered}
 N(b_1)\cap N(b_2)=\{u\},\qquad
 N(a_1)\cap N(a_2)=\{v\},\\
                     |E(\{a_1,a_2\},\{b_1,b_2\})|\le1.
 \end{gathered}
\end{equation}
Indeed, any additional common neighbour, or two edges between the
two pairs, would produce parallel copies that survive both
bicontractions.

The degree-two vertices of $H$ form a subset $Z$ of these four
vertices. Same-shore members of $Z$ have disjoint neighbourhoods
in $H$, and $H[Z]$ has at most one edge. If $H[Z]$ is edgeless,
the centre-and-two-neighbour triples for all $z\in Z$ are
pairwise disjoint. Same-shore triples could intersect only in a
common neighbour, excluded by~\eqref{eq:strictconditions};
opposite-shore triples could intersect only if their centres
were adjacent, excluded by the present case. Their simultaneous
bicontractions merge two
vertices of degree at least three and create degree at least four.
This is already the retract.

If $xy$ is the single edge of $H[Z]$, bicontract $x$ first. Its
other neighbour $z$ has degree at least three; the merger of $y,z$
has degree $d_H(z)\ge3$. No remaining centre can meet this
triple: same-shore neighbourhoods are disjoint, and there is no
second edge in $H[Z]$. The remaining triples are disjoint from
one another as well. Bicontract them to obtain the retract, with
no new degree-two vertex.

In either case, each vertex of $J$ is either an unchanged vertex
$x$ of $H$ or is obtained by identifying two vertices $x,y$ of
$H$. Its degree is respectively $d_H(x)$ or
$d_H(x)+d_H(y)-2$. The latter is at most $2\Delta(G)-2$, while
an unchanged vertex has degree at most $\Delta(G)\le2\Delta(G)-2$.
This proves
\eqref{eq:strictmax}. Parallel copies in $J$ have been retained;
we have not assumed that this cofactor retract is simple.
\end{proof}

\begin{lemma}\label{lem:index2}
Let $G$ be a simple brace of rank $r\ge R$ with
$m(G)<(r-1)^2$. For every strictly thin edge $uv$ of index two,
\begin{equation}\label{eq:index2gap}
                            m(G-u-v)\ge4r-6.
\end{equation}
\end{lemma}
\begin{proof}
Theorem~\ref{thm:degree} gives $\Delta(G)<r/2-4$. Suppose instead
that $h=m(G-u-v)\le4r-7$. Use Lemma~\ref{lem:strictfibres} to
obtain a retract $J$ with $\kappa\le4$ bicontractions. Its rank
is $q=r-1-\kappa\ge r-5\ge7$, so~\eqref{eq:clmlinear}
supplies the lower bound in
\[
                   0\le B:=m(J)-4q+4\le4\kappa+1\le17.
\]
If $J$ has a unit pair, \eqref{eq:hubdegree} gives
\[
 \Delta(J)\ge q-1-\lfloor B/4\rfloor
            \ge r-2-2\kappa\ge r-10.
\]
Here $B\le4\kappa+1$ gives $\lfloor B/4\rfloor\le\kappa$;
retaining the floor is essential for this comparison.
But~\eqref{eq:strictmax} gives
$\Delta(J)\le2\Delta(G)-2<r-10$, a contradiction.

If $J$ has no unit pair, Lemma~\ref{lem:dichotomy} with budget
$17$ applies because $q\ge r-5\ge22^2$ and
$m(J)\le4q-4+17$. It yields
$|E(J)|-3q\le3\cdot22^2+34=1486$. Both endpoints
of $uv$ are cubic, so $|E(J)|=|E(G)|-5-2\kappa$.
Subtracting $3q=3r-3-3\kappa$ gives $|E(J)|-3q=S-2+\kappa$;
hence
$S=|E(G)|-3r\le1488-\kappa\le1488$.
Apply Lemma~\ref{lem:entropy} to the original simple graph $G$,
not to its possibly nonsimple retract $J$. It gives
\[
                 m(G)\ge(4/3)^{r/497}>(r-1)^2,
\]
again a contradiction. The strict comparison holds for
$r\ge10^6$: use $\log(4/3)>1/4$ and
$r/1988>2\log(r-1)$ at $r=10^6$; the difference increases
thereafter. This proves~\eqref{eq:index2gap}.
\end{proof}

\begin{proof}[Proof of Theorem~\ref{thm:main}]
Set $Q(r)=(r-1)^2$ and $D(G)=Q(r)-m(G)$. Suppose a strict
counterexample $G_0$ has rank $r_0\ge R^2$. At any strict
counterexample $G$ of rank $r\ge R$, Theorem~\ref{thm:degree}
gives $\Delta(G)<r/2-4$. It is not a biwheel or an extended
biwheel. Nor is it a cubic exceptional brace: Lemma~\ref{lem:entropy}
with $S=0$ gives $m(G)\ge(4/3)^r>Q(r)$ at these ranks.
Theorem~\ref{thm:strictthin} therefore supplies an edge $uv$
of index $i\in\{0,1,2\}$.

Let $K=\ret(G-uv)$, a simple brace of rank $r-i$, and put
$h=m(G-u-v)$. The matchings omitting $uv$ are counted by
$m(G-uv)=m(K)$; deleting $uv$ from a matching containing it
gives a bijection with the $h$ matchings of $G-u-v$. Thus
\begin{equation}\label{eq:descent}
 m(G)=m(K)+h,\qquad
 D(K)-D(G)=h-[Q(r)-Q(r-i)].
\end{equation}
If $i=0$, the deficit increases by $h\ge1$, while size decreases.
If $i=1$, write $d$ for the degree of the noncubic endpoint.
The initial cofactor has rank $r-1$ and at most $d+1$ degree-two
vertices, by the proof of Lemma~\ref{lem:cofactorbasic} with
endpoint degrees $d,3$. Its retract therefore has rank at least
$r-d-2\ge7$, so
Lemma~\ref{lem:cofactorbasic} gives
\[
               h\ge4r-4d-12>2r+4.
\]
Since $Q(r)-Q(r-1)=2r-3$, the deficit increases by at least eight.
Here $h$ is an integer, so $h>2r+4$ implies $h\ge2r+5$.
If $i=2$, Lemma~\ref{lem:index2} gives $h\ge4r-6$, while
$Q(r)-Q(r-2)=4r-8$; the deficit increases by at least two.

Thus every reduction preserves a strict counterexample and increases
its deficit by at least its rank loss: the bounds are respectively
$1,8,2$ for rank losses $0,1,2$. The pair (rank, size)
decreases lexicographically. Continue while rank is at least $R$.
The process must reach a rank $k<R$, since a further step is always
available at or above $R$. Telescoping~\eqref{eq:descent} gives
\[
                         D\ge1+r_0-k.
\]
The initial deficit is a positive integer, hence at least one;
the rank losses sum to $r_0-k$. Index-zero steps only increase
the deficit further and do not affect that sum.
At the terminal graph we need only the elementary bound $m\ge1$,
which gives $D\le(k-1)^2-1$. Consequently
\[
                         r_0\le k^2-k-1<R^2,
\]
a contradiction. This establishes the lower bound. Biwheels attain
it at every rank $r\ge4$, proving the asserted eventual equality.
\end{proof}

\section*{Declaration on the use of AI}
The authors used AI tools to assist in discussing proof 
strategies, checking proofs, and improving the exposition. The 
authors take full responsibility for the mathematical arguments, 
results, and conclusions, all of which were carefully reviewed 
and verified by them.

\end{document}